\documentclass{amsart}
\pdfoutput=1
\usepackage{amssymb,amsmath,amsthm}
\usepackage{mathrsfs}
\usepackage{fullpage}
\usepackage{tikz}
\usetikzlibrary{calc}
\begin{document}
\title{An index theorem of Callias type for pseudodifferential operators}
\author{Chris Kottke}
\address{Massachusetts Institute of Technology\\ Department of Mathematics\\ Cambridge, MA 02139}
\curraddr{Brown University\\ Department of Mathematics\\ Providence, RI 02912}
\email{ckottke@math.brown.edu}
\subjclass[2010]{Primary 58J20; Secondary 19K56, 58J40}
\keywords{index theorem, scattering pseudodifferential operator, Dirac operator, scattering manifold, asymptotically conic manifold, 
asymptotically locally Euclidean manifold}
\newtheorem{thm}{Theorem}
\newtheorem{prop}[thm]{Proposition}
\newtheorem{cor}[thm]{Corollary}
\newtheorem{lem}[thm]{Lemma}
\newtheorem*{thm-}{Theorem}
\theoremstyle{remark}
\newtheorem*{rmk}{Remark}
\newtheorem*{rem}{Remark}

\newcommand\e{\epsilon}
\newcommand\pa{\partial}
\newcommand\id{\mathrm{Id}}
\newcommand\pt{\mathrm{pt}}
\newcommand\N{\mathbb{N}}
\newcommand\C{\mathbb{C}}
\newcommand\R{\mathbb{R}}
\newcommand\Z{\mathbb{Z}}

\newcommand{\pair}[1]{\left\langle #1\right\rangle}
\newcommand\set[1]{\left\{ #1 \right\}}
\newcommand\abs[1]{\left|#1\right|}
\newcommand\norm[1]{\left\|#1\right\|}
\newcommand\parens[1]{\left(#1\right)}

\newcommand\wt[1]{\widetilde{{#1}}}

\newcommand\smallto{\rightarrow}
\renewcommand\to{\longrightarrow}
\renewcommand\mapsto{\longmapsto}

\newcommand\Diffeo{\mathrm{Diffeo}}
\newcommand\Diff{\mathrm{Diff}}
\newcommand\scV{\mathcal{V}_{\mathrm{sc}}}
\newcommand\scH{H_{\mathrm{sc}}}
\newcommand\bscH{H_{\mathrm{b,sc}}}
\newcommand\scP{\Psi_{\mathrm{sc}}}
\newcommand\scDiff{\mathrm{Diff}_{\mathrm{sc}}}
\newcommand\bV{\mathcal{V}_{\mathrm{b}}}
\newcommand\Ker{\mathrm{ker}}
\newcommand\Coker{\mathrm{coker}}
\newcommand\GL{\mathrm{GL}}
\newcommand\supp{\mathrm{supp}}
\newcommand\Sym{\mathrm{Sym}}
\newcommand\End{\mathrm{End}}
\newcommand\Aut{\mathrm{Aut}}
\newcommand\Hom{\mathrm{Hom}}
\newcommand\cV{\mathcal{V}}
\newcommand\cI{\mathcal{I}}
\newcommand\dirac[1]{{/}\!\!\!\!#1}
\newcommand\scT{{}^{\mathrm{sc}}T}
\newcommand\bT{{}^{\mathrm{b}}T}
\newcommand\scS{{}^{\mathrm{sc}}S}
\newcommand\scN{{}^{\mathrm{sc}}N}
\newcommand\Id{\mathrm{Id}}
\newcommand\ind{\mathrm{ind}}
\newcommand\topind{\mbox{\rm{top-ind}}}
\newcommand\dom{\mathrm{Dom}}
\newcommand\ran{\mathrm{Ran}}
\newcommand\Iso{\mathrm{Iso}}
\newcommand\ch{\mathrm{ch}}
\newcommand\Td{\mathrm{Td}}
\newcommand\Thom{\mathfrak{T}}
\newcommand\isym{\sigma_{\mathrm{int}}}
\newcommand\ssym{\sigma_{\mathrm{sc}}}
\newcommand\tsym{\sigma_{\mathrm{tot}}}
\newcommand\nisym{{}_m\sigma_{\mathrm{int}}}
\newcommand\nssym{{}_m\sigma_{\mathrm{sc}}}
\newcommand\ntsym{{}_m\sigma_{\mathrm{tot}}}
\newcommand\wn{\mathrm{wn}}
\newcommand\cl{\mathrm{c}\ell}
\newcommand\Cl{\mathbb{C}\ell}
\newcommand\tM{\widetilde M}
\newcommand\tP{\widetilde P}
\newcommand\tX{\widetilde X}
\newcommand\tY{\widetilde Y}
\newcommand\tV{\widetilde V}
\newcommand\talpha{\widetilde \alpha}
\newcommand\sspan{\mathrm{span}}

\begin{abstract}
We prove an index theorem for families of pseudodifferential operators generalizing those studied by C.\ Callias, N.\ Anghel and others.  Specifically, we consider
operators on a manifold with boundary equipped with an asymptotically conic (scattering) metric, which have the form $D + i\Phi$, where $D$ is elliptic 
pseudodifferential with Hermitian symbols, and $\Phi$ is a Hermitian bundle endomorphism which is invertible at the boundary and commutes with the 
symbol of $D$ there.  The index of such operators is completely determined by the symbolic data over the boundary.  We use the scattering calculus of R.\ Melrose 
in order to prove our results using methods of topological K-theory, and we devote special attention to the case in which $D$ is a family of Dirac operators, in 
which case our theorem specializes to give families versions of the previously known index formulas.
\end{abstract}

\maketitle

\section{Introduction}
In \cite{callias1978axial}, C.\ Callias obtained a formula for the Fredholm index of operators on odd-dimensional Euclidean space $\R^n$ having the form
\[
	P = \dirac{D}\otimes 1 + i \otimes \Phi(x) : C^\infty_c(\R^n; V'\otimes V'') \to C^\infty_c(\R^n; V'\otimes V''),
\]
where $\dirac{D}$ is a self-adjoint spin Dirac operator (associated to any connection with appropriate flatness at infinity), $\Phi$ is a Hermitian matrix-valued
function which is uniformly invertible off a compact set (representing a Higgs potential in physics), and $V'$ and $V''$ are trivial vector bundles.
According to his formula, the index depends only on a topological invariant of $\Phi$ restricted to $S^{n-1}$, the sphere at infinity.
In a following paper \cite{bott1978some}, R.\ Bott and 
R.\ Seeley interpreted Callias' result in terms of a symbol map (equivalent to the total symbol defined below)
$\sigma(P) : S^{2n-1} = \pa\parens{\overline{T^\ast \R^n}} \to \End(V'\otimes V'')$ and point out that the resulting index formula has the form of the product of the
Chern characters of $\sigma(\dirac{D})$ and $\Phi$, each integrated over a copy of $S^{n-1}$.

The Fredholm index of Dirac operators coupled to skew-Hermitian nonscalar potentials on various odd dimensional manifolds was subsequently studied by 
several authors, culminating in a result in \cite{anghel1993index} by N.\ Anghel (also obtained independently by J.\ R{\aa}de in \cite{rade1994callias},
and U.\ Bunke in \cite{bunke1995k} who proved a $C^\ast$ equivariant version of the theorem applicable in particular to families of Dirac operators)
for operators of the above type on arbitrary odd-dimensional, complete Riemannian manifolds .  Under suitable conditions 
on $\dirac{D}$, a Dirac operator associated to a vector bundle $V \to X$, and on the potential 
$\Phi \in C^\infty(X; \End(V))$, Anghel proved that 
\[
	\ind(\dirac{D} + i\Phi) = \ind(\dirac{\pa}^+_+)
\]
where $\dirac{\pa}^+_+$ is a related Dirac operator on a hypersurface $Y \subset X$, representing a suitable ``infinity'' in $X$.
The proofs of these results depend on the fact that $\dirac{D}$ is a Dirac operator -- Callias' original proof uses local trace formulas
of the integral kernels, Anghel and Bunke use the relative index theorem of Gromov and Lawson \cite{gromov1983positive}, and R{\aa}de uses elliptic boundary conditions
analogous to the Atiyah-Patodi-Singer conditions to preserve the index under various cutting and gluing procedures.

In this paper we shall determine the index of Callias-type operators via methods in topological K-theory, in the spirit of \cite{atiyah1968index_I_III}
and \cite{atiyah1971index_IV}, using R.\ Melrose's calculus of scattering pseudodifferential operators \cite{melrose1994spectral}.  In particular, this allows 
us to consider a class of {\em pseudodifferential} Callias-type operators, which we dub Callias-Anghel operators, and to obtain a families version 
of the index theorem with little additional effort.  Our result applies also to even-dimensional manifolds, where Callias-Anghel operators which are not of Dirac 
type may indeed have nontrivial index.

A Callias-Anghel operator $P$ on a manifold $X$ with boundary $\pa X$ (thought of as ``infinity'') has the form $P = D + i\Phi$, where $D \in \scP^m(X;V)$ is an 
elliptic scattering pseudodifferential operator with Hermitian symbols, and $\Phi \in C^\infty(X; \End(V))$ is a compatible potential, meaning that 
$\Phi_{|\pa X}$ is Hermitian, invertible and commutes with the symbol of $D$.  Our main result, proved in section \ref{S:results}, is the following.
\begin{thm-}
A Callias-Anghel operator $P = D + i\Phi$ extends to a Fredholm operator on naturally defined Sobolev spaces, with index
\[
	\ind(P) = \int_{S^\ast_{\pa X} X} \ch (V^+_+)\cdot\pi^\ast\Td(\pa X).
\]
where $V^+_+ \subset \pi^\ast V \to S^\ast_{\pa X} X$ is the bundle of jointly positive eigenvectors of $\sigma(D)$ and $\pi^\ast \Phi$ on the cosphere bundle of $X$
over $\pa X$.
\end{thm-}
Note that $V^+_+$ is well-defined since $\sigma(D)$ and $\pi^\ast \Phi$ are Hermitian and commute, hence are jointly diagonalizable.

The essence of our proof is to note that the index of $D + i\Phi$ is determined topologically by the symbolic data, which is shown to be trivial over the 
interior of $X$ and completely determined by the coupling between Hermitian (from $\sigma(D)$) and 
skew-Hermitian (from $i\Phi$) terms on the cosphere bundle over infinity ($S^\ast_{\pa X} X$).  Indeed, once the problem is properly formulated, the proof is 
a straightforward computation in K-theory.

In using the scattering calculus of pseudodifferential operators, we restrict ourselves to 
asymptotically conic, or ``scattering'' manifolds which are compact manifolds with boundary, equipped with metrics of the form
\[
	\frac{dx^2}{x^4} + \frac{h}{x^2}, \quad\text{$h_{|\pa X}$ a metric on $\pa X$},
\]
where $x$ is a boundary defining function.  Indeed, in order to have some reasonable class of pseudodifferential operators at our disposal, it is necessary
to restrict the asymptotic geometry in some way, and our choice is motivated by the following considerations.
\begin{itemize}
\item While this class of manifolds is geometrically more restricted than the general complete Riemannian manifolds considered by Anghel and others, the conditions 
for operators to be Fredholm on these spaces are much {\em less} restrictive and more easily verified in practice.  Correspondingly, we need to assume less 
about $D$ and $\Phi$ to obtain our result.  We discuss a connection between our setup and the one considered by Anghel in section \ref{S:connection}, and 
expect that using scattering models in the context of certain other noncompact index problems (essentially situations in which the Fredholm data is sufficiently 
local near infinity) may be possible.
\item The symbolic structure of the scattering calculus has a very simple interpretation in terms of topological K-theory, permitting us to utilize a powerful families
index theorem (Theorem \ref{T:AS_prelim}) derived from \cite{atiyah1971index_IV}.
\item The author's work on this subject was motivated by his thesis work on $SU(2)$ monopole moduli spaces over asymptotically conic manifolds, where the 
dimension of the moduli space is given by the index of a Callias-Anghel type operator.  In that case the potential $\Phi$
may have (constant rank) null space over $\pa X$; however, the index problem for such an operator, which will be the subject of a subsequent paper, may nonetheless 
be reduced to one of the type considered here.
\end{itemize}

We begin with a brief introduction to the scattering calculus in section \ref{S:scattering}, culminating with the proof of the index theorem for 
families of scattering operators.  We introduce Callias-Anghel type operators in section \ref{S:callias} and prove that they extend to Fredholm operators.  Section
\ref{S:reduction} is the heart of our result, and consists of the reduction of the symbol by homotopy to the corner $S^\ast_{\pa X}X$ of the total space 
$\pa(\overline{T^\ast X})$ in K-theory; it is entirely topological in nature.  We present our results in section \ref{S:results}, with a particular 
analysis of the important case of Dirac operators.  Finally, we discuss the relation to previous results in section \ref{S:connection}.

The author would like to thank his thesis advisor Richard Melrose for his support and guidance, and also Pierre Albin for many helpful conversations.

\section{The Scattering Calculus}\label{S:scattering}
We briefly recall the important elements of the scattering calculus of pseudodifferential operators.  A basic reference for the material in this section is 
\cite{melrose1994spectral}, and more generally \cite{melrose-differential}.  By a {\em scattering manifold} we shall mean 
a compact manifold $X$ with boundary, typically equipped with an exact scattering metric as defined below. We refer to 
$\pa X$ as ``infinity.''  

\subsection{Structure algebra and bundles}\label{S:scattering_structure}
Given a compact manifold $X$ with boundary, the algebra of {\em scattering vector fields} $\scV(X)$ is a Lie subalgebra of the algebra $\cV(X)$ of vector fields
defined by  
\[
	\scV(X) = x\bV(X).
\]
where $x$ is a boundary defining function ($x \geq 0,\, x^{-1}(0) = \pa X,\, dx_{|\pa X} \neq 0$), and $\bV(X)$ is the subalgebra of  vector fields tangent to 
the boundary:
\[
	\bV(X) = \set{V \in \cV(X)\;;\; Vx \in xC^\infty(X)}.
\]
Both $\bV(X)$ and $\scV(X)$ are independent of the choice of $x$.
In a coordinate neighborhood near the boundary, $\scV(X)$ is spanned by $x^2\pa_x, x\pa_{y_1},\ldots, x\pa_{y_n}$ where $y_1,\ldots,y_n$ are local coordinates on 
$\pa X$.  

Just as $\cV(X) = C^\infty(X,TX)$ is the space of sections of the vector bundle $TX$, $\scV(X)$ and $\bV(X)$ are the spaces of sections of the
{\em scattering and b tangent bundles} $\scT X$ and $\bT X$, respectively.  In local coordinates $x,y_1,\ldots,y_n$ near $\pa X$, bases 
for $\scT_p X$ and $\bT_p X$ at a point $p$ are given by
\[
	\scT_p X = \sspan_\R \set{x^2 \pa_x, x\pa_{y_1},\ldots, x\pa_{y_n}}, \quad \bT_p X = \sspan_\R \set{x\pa_x,\pa_{y_1},\ldots,\pa_{y_n}}.
\]
The {\em scattering and b cotangent bundles} $\scT^\ast X$ and $\bT^\ast X$ are the dual bundles to $\scT X$ and $\bT X$, with bases in local coordinates
given by
\[
	\scT_p^\ast X = \sspan_\R \set{\frac{dx}{x^2}, \frac{dy_1}{x},\ldots, \frac{dy_n}{x}}, \quad \bT_p^\ast X = \sspan_\R \set{\frac{dx}{x}, dy_1,\ldots,dy_n}.
\]
While $T^\ast X$, $\bT^\ast X$ and $\scT^\ast X$ (or $T X$, $\bT X$ and $\scT X$) are isomorphic in the interior of $X$, they are not canonically so at the boundary 
(any identification over $\pa X$ depends on a choice of boundary defining function $x$).  However they are homotopy equivalent as vector bundles, so for topological 
purposes such as index computations, the difference is often unimportant.  The reader unfamiliar with the scattering calculus may mentally replace $\scT$ by $T$ 
without much harm.

The natural metrics to consider in the context of scattering operators are those which are smooth sections not of $\Sym^2(T^\ast X)$ but of $\Sym^2(\scT^\ast  X)$.
We will further restrict consideration to the case of so-called {\em exact} scattering metrics, which have the form
\[
	g = \frac{dx^2}{x^4} + \frac{h}{x^2},
\]
for some choice of a boundary defining function $x$, and where $h$ restricts to a metric on the compact manifold $\pa X$.  Scattering operators (defined below) naturally
extend to operators on Sobolev spaces associated to such metrics.

In preparation for section \ref{S:dirac}, we define {\em scattering (respectively b) connections} on a vector bundle $V \to X$ as covariant derivatives
$$
	\nabla : C^\infty(X; V) \to C^\infty(X; \scT^\ast X\otimes V)\quad \parens{\text{resp.\ } \nabla : C^\infty(X; V)\to C^\infty(X; \bT^\ast X\otimes V)},
$$
which satisfy the Leibniz condition $\nabla(f \cdot s) = df \otimes s + f \cdot \nabla s$, where $f \in C^\infty(X)$ and $s \in C^\infty(X; V)$.
Here we are taking the image of the one-form $df$ in either $C^\infty(X; \scT^\ast X)$ or $C^\infty(X; \bT^\ast X)$, which is possible since
there are natural bundle maps $T^\ast X \to \bT^\ast X \to \scT^\ast X$ induced by the inclusions 
$\scV(X) \subset \bV(X) \subset \cV(X)$.

From the maps $T^\ast X \to \bT^\ast X \to \scT^\ast X$, we also see that true connections (which is what we shall call connections in the ordinary sense) extend
naturally to b and scattering connections, and that b connections similarly extend to scattering connections.  If a scattering connection $\nabla$ is obtained from 
such a b connection, we will say that it is the {\em lift of a b connection}.

A choice of boundary defining function gives a natural product structure $\pa X \times [0,1)_x$ on a neighborhood of the boundary and hence a way to extend vector
fields from $\pa X$ into the interior.  In this way, true and b connections restrict to connections $\nabla_{|\pa X}$ on the boundary.  Note 
however that a scattering connection does not naturally restrict to a connection on $\pa X$ unless it is the lift of a true or b connection.  For instance, 
given an exact scattering metric $g$, the Levi-Civita connection $\nabla^{\mathrm{LC}(g)}$ is the lift of a b connection\footnote{Note that $\nabla^{\mathrm{LC}(g)}$
is {\em not}, however, the lift of the Levi-Civita connection associated to $x^2\,g$.} which restricts to $\nabla^{\mathrm{LC}(h)}$ at $\pa X$.

\subsection{Operators and symbols}\label{S:scattering_operators}
The algebra of scattering differential operators acting on sections of a vector bundle $V$ is just the universal enveloping algebra of 
$\scV(X)\otimes C^\infty(X;\End(V))$ over $C^\infty(X)$; for a given $k \in \N_0$,
\[
	\scDiff^k(X;V) = \set{\sum_{0 \leq l \leq k} c_l\, \nabla_{V_1} \cdots \nabla_{V_l}\;;\; V_i \in \scV(X), c_l \in C^\infty(X; \End(V))}.
\]
and $\scDiff^\ast (X;V)$ forms a filtered algebra of operators on $C^\infty(X;V)$.

This algebra can be ``microlocalized'' to produce an algebra of {\em scattering pseudodifferential operators} acting
on sections of $V$, denoted
$\scP^m (X;V), m \in \R$, by constructing their Schwartz kernels on an appropriately blown up version of the space $X^2$ (see \cite{melrose1994spectral} for details).

Given $D \in \scP^m(X;V)$, we have an {\em interior symbol} map analogous to the usual principal symbol,
\[
	\isym(D) : \scT^\ast X \to \End(\pi^\ast V),
\]
where\footnote{We will use $\pi$ to denote the projection for various bundles related to $\scT^\ast X$, such as the scattering cosphere bundle $\scS^\ast X$ and the 
radially compactified scattering cotangent bundle $\overline{\scT^\ast X}$.  The appropriate domain will be clear from context, and no confusion should arise.} 
$\pi : \scT^\ast X \to X$.
In addition to this symbol, we have a {\em boundary} or {\em scattering symbol}
\[
	\ssym(D) : \scT^\ast_{\pa X} X \to \End(\pi^\ast V)
\]
which, for a differential operator in local coordinates, takes the form
\[
	D = \sum_{\abs{\alpha} + j \leq m} a_\alpha(x,y) (x^2\pa_x)^j (x\pa_y)^\alpha \; \implies\; 
	\ssym(D)(0,y,\zeta,\eta) = \sum_{\abs{\alpha} + j \leq m} a_\alpha(0,y) \zeta^j \eta^\alpha.
\]
(Note that the boundary symbol involves the sum over all orders in the operator, whereas the interior symbol only involves the top order $\abs{\alpha} + j = m$.)

Both symbols have asymptotic growth/decay of order $ \leq m$ along the fibers, and they satisfy the compatibility condition that, asymptotically, 
\[
	\isym (D) (p,\xi) \sim \ssym (D) (q,\xi)\quad \text{ as $p \smallto q$, $\abs{\xi} \smallto \infty$},
\]
where $p \in X$, $q \in \pa X$, $\xi \in \scT^\ast_p X$.  

We will restrict ourselves to so-called ``classical'' operators whose symbols have asymptotic expansions in terms of $\abs{\xi}^{m-k}, k \in \N$, as 
$\abs{\xi}\smallto\infty$.  Then for 0th order operators, we can regard the interior symbol as a map
\[
	\isym(D) : \scS^\ast X \to \End(\pi^\ast V),
\]
where $\scS^\ast_p X$ is the boundary of the radially compactified fiber $\overline{\scT^\ast_pX}$, and the value of $\isym(D)$ is obtained by taking the limit of the 
leading term in the asymptotic expansion.  Similarly, we extend $\ssym(D)$ to a map
\[
	\ssym(D) : \overline{\scT^\ast_{\pa X} X} \to \End(\pi^\ast V),
\]
again using radial compactification of the fibers.  Since $D$ is 0th order, both symbols are bounded, asymptotic compatibility is just
equality of the limits, and  we can combine the two symbols into a continuous {\em total symbol}
\[
	\tsym(D) : \pa (\overline{\scT^\ast X}) \to \End(\pi^\ast V),
\]
where $\overline{\scT^\ast X}$ is the total space of the compactified scattering cotangent bundle.  It is a manifold with corners, with boundary 
$\pa(\overline{\scT^\ast X})$ consisting of both $\scS^\ast X$ and $\overline{\scT^\ast_{\pa X} X}$, which intersect at the corner $\scS^\ast_{\pa X} X$ 
(see Figure \ref{F:totalspace}).

We can produce a total symbol in the general case as follows.  For every $m \in \R$, we construct a trivial real line bundle $N_m \to \overline{\scT^\ast X}$ whose 
bounded sections consist of functions with asymptotic growth/decay of order $m$.  
Given a scattering metric, a trivialization over the interior is given by the section $\abs{\xi}^m$, that is
\[
	\scT^\ast X \times \R \to N_m : \parens{(p,\xi),t} \stackrel{\cong}{\mapsto} (p,\xi,t\abs{\xi}_p^m).
\]
Symbols of $m$th order operators define bounded sections of $N_m$, which take limiting values at the boundary as above, and we define the {\em renormalized symbols}
as 
\[
	\nisym(D) : \scS^\ast X \to N_m\otimes\End(\pi^\ast V), \quad \nssym(D) : \overline{\scT^\ast_{\pa X} X} \to N_m\otimes\End(\pi^\ast V).
\]
We combine these to obtain the renormalized total symbol
\[
	\ntsym(D) : \pa(\overline{\scT^\ast X}) \to N_m\otimes \End(\pi^\ast V).
\]

Given $D \in \scP^m(X;V)$, we say $D$ is {\em elliptic} when its interior symbol $\isym(D)$ is invertible, as usual.  Elliptic scattering operators satisfy the usual
elliptic regularity conditions, but in general fail to be Fredholm as operators on any natural Sobolev spaces.  $D$ is said to be {\em fully elliptic} 
if {\em both} its interior symbol $\isym(D)$ and its boundary symbol $\ssym(D)$ are everywhere invertible.  This is equivalent 
to invertibility of the renormalized total symbol $\ntsym(D)$, since invertibility does not depend on the chosen trivialization of $N_m$.   If $D$ is fully elliptic, 
it has a unique extension from an operator on $C^\infty_c(X; V)$ to a bounded, Fredholm operator on weighted{\em scattering Sobolev spaces}:
\[
	D : x^\alpha \scH^{m+k}(X;V) \to x^\alpha \scH^k(X;V) \quad \text{is Fredholm}
\]
for all $\alpha \in \R$, where for\footnote{It is straightforward to define scattering Sobolev spaces of all real orders in terms of pseudodifferential operators, 
but we restrict ourselves here to the case of non-negative integer $k$ for simplicity.} $k \in \N_0$,
\[
	x^\alpha \scH^k(X;V) = \set{v = x^\alpha u \;;\; u \in L^2(X;V),\text{ and } Pu \in L^2(X;V)\text{ for all } P \in \scDiff^k(X;V)}.
\]

\begin{rem}
Note that in discussing the total symbols of pseudodifferential operators, we use the notation $\pi : \pa(\overline{\scT^\ast X}) \to X$ to denote the 
projection.  This is not a proper fiber bundle, as the fiber over the interior is a sphere, $\scS^\ast_p X$, while the fiber over a boundary point 
is the (radially compactified) vector space $\overline{\scT^\ast_p X}$.  Nevertheless, the notation is convenient.
\end{rem}

\subsection{Families of operators}\label{S:scattering_families}
Below we shall consider families of scattering pseudodifferential operators, for which we use the following notation.  Suppose $X$ has the structure
of a fiber bundle $X \to Z$, where $Z$ is a compact manifold without boundary, and such that the fiber is a manifold $Y$ with boundary $\pa Y$.  We 
use the notation $X/Z := Y$ to denote the fiber, though there is no real such quotient.  Thus $X$ has boundary $\pa X$ which itself fibers over $Z$, with
fiber $\pa Y$.  $X$ is associated to a principal $\Diffeo(Y)$-bundle $\mathcal{P} \to Z$, from which we derive additional associated bundles.  


Suppose we are given a metric on $X$ which restricts to a fixed exact scattering metric on each fiber, for instance by taking a scattering metric on the total space.  
A {\em family of scattering operators on $X \to Z$},
is an operator acting on sections of a vector bundle\footnote{Any vector bundle $V\to X$ can be exhibited as a family of vector bundles
$V = \mathcal{P}\times_{\Diffeo(Y)} W$, where $W \to Y$ is a fixed vector bundle with the same rank as $V$.} 
$V \to X$ which is scattering pseudodifferential in the fiber directions, and smoothly varying in the base.  It is properly defined as a section of the bundle
\[
	\scP^m(X/Z; V) = \mathcal{P} \times_{\Diffeo(Y)} \scP^m(Y;W) \to Z,
\]
where $V$ and $W$ are related by $V = \mathcal{P} \times_{\Diffeo(Y)} W$.

In the simple case that $X$ is a product, $X = Y \times Z$, $Z$ is just a smooth parameter space for the operators, and we recover the case of a single operator 
by taking $Z = \pt$, $X/Z = X = Y$. 

For a family $D \in \scP^m(X/Z; V)$ of operators, the symbol maps have domain $\scT^\ast (X/Z)$, which is the vertical scattering cotangent bundle
\[
	\scT^\ast (X/Z) = \mathcal{P} \times_{\Diffeo(Y)} \scT^\ast Y \to Z,
\]
with fibers isomorphic to the scattering cotangent bundle $\scT^\ast Y$ of the fiber.  The renormalized total symbol is a map
\[
	\ntsym(D) : \pa\parens{\overline{\scT^\ast (X/Z)}} \to N_m\otimes\End(\pi^\ast V)
\]
where now everything is fibered over $Z$, and $\pa(\overline{\scT^\ast (X/Z)})$ has fibers isomorphic to $\pa(\overline{\scT^\ast Y})$.  Note that
$\pi : \pa(\overline{\scT^\ast (X/Z)}) \to X$ is a family of projections modeled on $\pi : \pa(\overline{\scT^\ast Y}) \to Y$, to which the remark at the end
of section \ref{S:scattering_operators} applies.

As in the case of ordinary pseudodifferential operators, a family $D$ of Fredholm operators\footnote{The Fredholm property in the families setting is with respect to
families of scattering Sobolev spaces $\scH^\ast(X/Z; V) = \mathcal{P}\times_{\Diffeo(Y)} \scH^\ast(Y; W)$.\label{footnote:sobolev}} over $Z$ has an 
index $\ind(D) \in K^0(Z)$ given by 
\[
	\ind(D) = [\Ker(D)] - [\Coker(D)] \in K^0(Z),
\]
which is well-defined by a stabilization procedure and Kuiper's theorem \cite{lawson1989spin}.

\begin{figure}[htb]
\begin{center}
\begin{tikzpicture}[scale=2]
	\clip	(-3,2.25) rectangle (1.5, -0.5);

	\path 	(-2,0) coordinate (a)
		(-1,0) coordinate (b)
		(0,0) coordinate (c)
		(1,0) coordinate (d)
		(2,0) coordinate (e)
		(-2,2) coordinate (a1)
		(-1,2) coordinate (b1)
		(0,2) coordinate (c1)
		(1,2) coordinate (d1)
		(2,2) coordinate (e1)
		(-2,-1) coordinate (a2)
		(-1,-1) coordinate (b2)
		(0,-1) coordinate (c2)
		(1,-1) coordinate (d2)
		(2,-1) coordinate (e2);
	\shadedraw [shading=radial, inner color=white, outer color=gray, thick] (a1) rectangle (e2);
	\draw [thick,dashed] 
		(b1) -- (b2)
		(c1) -- (c2)
		(d1) -- (d2);
	\draw [thick] (a) -- (e);
	\fill (a1) circle (1pt) node[above left] {$\scS^\ast_{\pa X} (X/Z)$};
	\node [below left] at (c) {$X$};
	\node [above left] at (c1) {$\scS^\ast (X/Z)$};
	\node [left] at ($(c)!0.5!(c1)$) {$\overline{\scT^\ast (X/Z)}$};
	\node [below left] at (a) {$\pa X$};
	\node [left] at ($(a)!0.5!(a1)$) {$\overline{\scT^\ast_{\pa X} (X/Z)}$};
\end{tikzpicture}	
\caption{The total space $\overline{\scT^\ast (X/Z)}$ and its boundary}
\label{F:totalspace}
\end{center}
\end{figure}
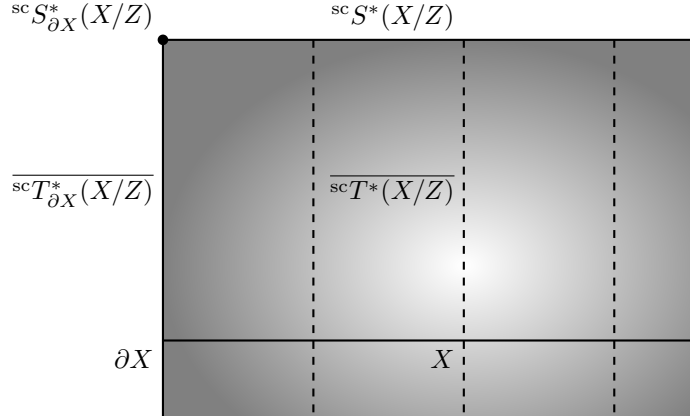

\subsection{The index theorem}\label{S:scattering_index}
The following theorem is one of the primary reasons for using the scattering calculus in our treatment.  Among calculi of pseudodifferential operators on 
manifolds with boundary, the scattering calculus is particularly simple since its boundary symbols are {\em local}\footnote{The trade off is that the condition of 
full ellipticity in the scattering calculus is stronger than the corresponding condition in calculi with less local boundary data.}, and hence give well-defined 
elements in the compactly supported topological $K$-theory of $\scT^\ast X$.  In particular, this allows for the index to be computed by a reduction to the 
Atiyah-Singer index theorem for compact manifolds (\cite{atiyah1971index_IV}, \cite{atiyah1968index_I_III}).  
A proof of Theorem \ref{T:AS_prelim} can be found in \cite{melrose1995geometric} and \cite{melrose2004families}, and a more explicit version of the cohomological formula
(an extension of Fedosov's formula for the classical index theorem) is obtained in \cite{albin2008relative}.

As our applications are to self-adjoint
operators with skew-adjoint potentials, the domain and range bundles of our operators will always be the same, which permits us to write the index formula
below in terms of the {\em odd} Chern character of the total symbol, which in this case defines an element of the odd K-group $K^1(\pa(\overline{\scT^\ast X}))$.  

First let us introduce the notation we use for K-theory.  As usual, we write elements in even K-theory as formal differences of 
vector bundles up to equivalence and stabilization, 
\[
	[V] - [W] \in K^0(M)
\]
and use the notation
\[
	[V,W,\sigma] \in K^0(M, N)
\]
for relative classes, where $\sigma : V_{|N} \stackrel{\cong}{\to} W_{|N}$ is an isomorphism over $N$.  This applies in particular to K-theory with 
{\em compact support} $K^0_c(M) = K^0(M, \infty)$, which is just K-theory relative to infinity with respect to any compactification 
(the one point compactification $M \cup \set{\infty}$ is typically used, though for vector bundles $V$, we use the fiberwise radial compactification $V \cup SV$).

Odd K-theory is represented by homotopy classes of maps $M \to \lim_{n\smallto\infty} \GL(n)$, though we use the notation
\[
	[V, \sigma] \in K^1(M)
\]
as shorthand for the element $[V\oplus V^{\perp}, \sigma\oplus\id] \in K^1(M)$, where $V\oplus V^{\perp} \cong M\times \C^N$, so 
$\sigma\oplus \id : M \to \GL(N)$.  In particular, an element $[V,V,\sigma] \in K^0(M,N)$ with identical domain and range bundles is the image of an element
$[V,\sigma] \in K^1(N)$ in the long exact sequence of the pair $(M,N)$.

Lastly, we define a topological index map for scattering pseudodifferential operators analogous to the classical one.  For $D \in \scP^m(X/Z; V,W)$ fully elliptic,
$[\pi^\ast V,\pi^\ast W, \ntsym(D)] \in K^0_c(\scT^\ast (\mathring{X}/Z))$ is well-defined\footnote{This involves choosing a trivialization of $N_m$, though the
element of K-theory obtained is independent of this choice.}, where $\mathring{X} = X \setminus \pa X$ is the interior of $X$ 
(so the compact support refers both to the fiber {\em and} base directions).  

Given an embedding $\mathring{X} \hookrightarrow \R^N \times Z$ of fibrations over $Z$, we have an induced, K-oriented embedding 
$\scT^\ast (\mathring{X}/Z) \to \R^{2N}\times Z$ into an even dimensional trivial Euclidean fibration.
We define $\topind(D)$ to be the image of $[\pi^\ast V,\pi^\ast W, \ntsym(D)]$ under the composition
\[
	K^0_c(\scT^\ast (\mathring{X}/Z)) \to K^0_c(N(\scT^\ast (\mathring{X}/Z))) \to K^0_c(\R^{2N}\times Z) \cong K^0(Z),
\]
where the first map is the Thom isomorphism onto the normal bundle of $\scT^\ast (\mathring{X}/Z)$ in $\R^{2N}\times Z$, the second is the pushforward 
with respect to the open embedding 
$N(\scT^\ast (\mathring{X}/Z)) \hookrightarrow \R^{2N} \times Z$ in compactly supported K-theory\footnote{Recall that, while cohomology theories 
(K-theory in particular), are contravariant, there is a limited form of covariance with respect to open embeddings 
in any compactly supported theory.  If $i: O \hookrightarrow M$ is an open embedding, we obtain a pushforward map $i_\ast : K^\ast_c(O) \to K^\ast_c(M)$
via the quotient map $M^+ / \infty \to M / (M \setminus O) \cong O^+ / \infty$, where $M^+$ denotes the one point compactification $M \cup \set{\infty}$.}, 
and the last is the Bott periodicity isomorphism (equivalently, the Thom isomorphism for a trivial bundle).
That this is a well-defined map independent of choices follows exactly as in the classical case in \cite{atiyah1971index_IV}.

\begin{thm} \label{T:AS_prelim} \cite{melrose1995geometric}, \cite{melrose2004families}
Let $P \in \scP^m(X/Z; V)$ be a family of fully elliptic scattering pseudodifferential operators.  It is therefore a Fredholm family, with well-defined
index $\ind(P) \in K^0(Z)$, and 
\[
	\ind(P) = \topind(P).
\]
Furthermore, the Chern character of this index is given by the cohomological formula
\[
	\ch(\ind(P)) = p_!\parens{\ch(\tsym(P))\cdot \pi^\ast \Td(X/Z)},
\]
where $p_! : H^\mathrm{even}_c(\scT^\ast (\mathring{X}/Z)) \to H^\mathrm{even}(Z)$ denotes integration over the fibers and $\ch(\tsym(P))$
is shorthand for $\ch_\mathrm{even}([\pi^\ast V, \pi^\ast V, \ntsym(P)]) \in H^\mathrm{even}_c(\scT^\ast (\mathring{X}/Z))$.
\end{thm}

Since $[\pi^\ast V, \pi^\ast V, \ntsym(P)] \in K^0_c(\scT^\ast (\mathring{X}/Z))$ is in the image of $K^1(\pa(\overline{\scT^\ast(X/Z)}))$, we can reformulate
the above in terms of the odd Chern character as follows.  First, we define a generalized fiber integration map 
$q_! : H^\mathrm{odd}(\pa(\overline{\scT^\ast (X/Z)})) \to H^\mathrm{even}(Z)$, where $q$ is the composition 
$q : \pa(\overline{\scT^\ast (X/Z)}) \stackrel{\pi}{\to} X \to Z$.  Of course, $q$ is not properly a fibration as per the remark at the end of 
\ref{S:scattering_operators}; it really consists of two fibrations 
$q_1 : \scS^\ast (X/Z) \to Z$ and $q_2 : \overline{\scT^\ast_{\pa X} (X/Z)} \to Z$, with an identification of their common boundary, which is the fibration 
$\scS^\ast_{\pa X} (X/Z) \to Z$.  We define $q_!$ as the sum
\[
	q_! \mu = (q_1)_! \mu + (q_2)_! \mu
\]
pulling back $\mu \in H^\mathrm{odd}(\pa(\overline{\scT^\ast (X/Z)}))$ as appropriate in each of the summands; it is well-defined on cohomology since if $\mu = d\alpha$ 
is exact (or more generally fiberwise exact), 
\[
	(q_! d\alpha)(z) = \int_{\parens{\overline{\scT^\ast (X/Z)}}_z}d\alpha + \int_{\parens{\scS^\ast (X/Z)}_z}d\alpha = 
	\int_{\parens{\scS^\ast_{\pa X} (X/Z)}_z} \alpha - \int_{\parens{\scS^\ast_{\pa X} (X/Z)}_z} \alpha = 0
\]
by Stokes' Theorem, since $\overline{\scT^\ast (X/Z)}$ and $\scS^\ast (X/Z)$ share the common boundary $\scS^\ast_{\pa X} (X/Z)$ but with opposite orientation.

Now, since $q_!$ factors as the composition of the connecting map $H^{\mathrm{odd}}(\pa(\overline{\scT^\ast (X/Z)})) \to H^{\mathrm{even}}_c(\scT^\ast (\mathring{X}/Z))$
with the even fiber integration map $p_! : H^{\mathrm{even}}_c(\scT^\ast (\mathring{X}/Z)) \to H^{\mathrm{even}}(Z)$, and since the connecting maps in
even/odd cohomology and K-theory intertwine the even/odd Chern character maps, we obtain the following odd version of
the cohomological formula.

\begin{cor} \label{C:AS}
Let $P \in \scP^m(X/Z; V)$ be a family of fully elliptic scattering pseudodifferential operators as above.  Then
\[
	\ch(\ind(P)) = q_!\parens{\ch_{\mathrm{odd}}(\tsym(P))\cdot \pi^\ast \Td(X/Z)},
\]
where $\tsym(P)$ is short for $[\pi^\ast V, \ntsym(P)] \in K^1(\pa(\overline{\scT^\ast (X/Z)}))$, defined using any trivialization of $N_m$.

In the special case of a single operator $P \in \scP^m(X; V)$, the index is an integer $\ind(P) \in \Z$, and we have
\[
	\ind(P) = \int_{\pa(\overline{\scT^\ast X})} \ch_{\mathrm{odd}}(\tsym(P))\cdot \pi^\ast \Td(X).
\]
\end{cor}

\section{Callias-Anghel type operators}\label{S:callias}
We shall be concerned with pseudodifferential families $D$ whose symbols are Hermitian, coupled to skew-Hermitian potentials $i\Phi$.  It is actually only 
necessary that $i\Phi$ be skew-Hermitian at infinity, as well as satisfy some compatibility conditions with $D$.  This is more general than the operators
considered in the literature, and we will see in the section following this one why the index is only dependent on these conditions.

Let $V \to X$ be a family of Hermitian complex vector bundles associated to the family of scattering manifolds $X \to Z$.  We will denote the inner product
on $V$ by $\pair{\cdot,\cdot}$.  Let $D \in \scP^m(X/Z;V), m > 0$ be a family of elliptic (but not necessarily {\em fully} elliptic) 
scattering operators with Hermitian symbols, so
\[
	\isym(D) : \scS^\ast (X/Z) \to \GL(\pi^\ast V) \subset \End(\pi^\ast V)
\]
is Hermitian with respect to $\pair{\cdot,\cdot}$ and 
\[
	\ssym(D) : \scT^\ast_{\pa X} (X/Z) \to \End(\pi^\ast V)
\]
is Hermitian but not necessarily invertible.  Of particular interest later will be the case of a family of Dirac operators, for which 
$\ssym(D)(p,\xi) = i \cl(\xi)\cdot$ vanishes at the 0-section over $\pa (X/Z)$ and is therefore {\em never} fully elliptic.

Next let $\Phi$ be a section of $\End(V)$.  Motivated by physics, we refer to $\Phi$ as the {\em potential}.  We will 
assume $\Phi$ satisfies the following conditions over the boundary $\pa X$, which we shall dub {\em compatibility with $D$}.
\begin{enumerate}
\item $\Phi_{|\pa X}$ is Hermitian with respect to $\pair{\cdot,\cdot}$
\item $\Phi_{|\pa X}$ is invertible 
\item $\Phi_{|\pa X}$ commutes with the boundary symbol of $D$, that is
\[
	[\pi^\ast \Phi_{|\pa X}, \ssym(D)] = 0 \in \End(\pi^\ast V) \quad \text{ on $\overline{\scT^\ast_{\pa X} X}$}.
\]
\end{enumerate}
We refer to condition 3 as {\em symbolic commutativity}.

Given $D$ and a compatible potential $\Phi$, the {\em Callias-Anghel type operator}
\[
	P = D + i\Phi \in \scP^m(X/Z; V)
\]
is fully elliptic (and therefore Fredholm on appropriate spaces) by the following elementary lemma.
\begin{lem}\label{L:fundamental}
Let $\alpha$ and $\beta$ be Hermitian sections of the bundle $\End(V) \to M$, and suppose that, over a subset $\Omega \subset M$, we have
$[\alpha,\beta] = \alpha\beta - \beta\alpha  = 0 \in \Gamma(\Omega; \End(V))$.  If either $\alpha$ or $\beta$ is invertible over $\Omega$, then the
combination
\[
	\alpha + i\beta \in \Gamma(\Omega; \GL(V))
\]
is invertible over $\Omega$.  

In particular, if both $\alpha$ and $\beta$ are invertible over $\Omega$, then the combination
\[
	t\alpha + i\,s\beta \in \Gamma(\Omega; \End(V)) \text{ is invertible for all $(s,t) \neq (0,0) \in \R_+^2$}.
\]
\end{lem}
\begin{proof}
It suffices to consider an arbitrary fiber $V_p$, $p \in \Omega$.  By the assumption that $\alpha$ and $\beta$ are Hermitian, $\alpha(p)$ has purely real
eigenvalues while $i\beta(p)$ has purely imaginary ones.  Since $[\alpha,\beta] = 0$, there is a basis of $V_p$ in which $\alpha(p)$ and 
$\beta(p)$ are simultaneously diagonal; with respect to this basis $\alpha + i\beta$ acts diagonally with eigenvalues of the form $\lambda_j + i\mu_j$
with $\lambda_j, \mu_j \in \R$.  If either $\alpha$ or $\beta$ is invertible, then either $\lambda_j \neq 0$ or $\mu_j \neq 0$ for all $j$; therefore
$\lambda_j + i\mu_j \neq 0 \in \C$ and $\alpha + i\beta$ must be invertible.
\end{proof}

\begin{cor}
The family of scattering operators $P = D + i\Phi$ extends to a family of Fredholm operators\footnote{See the footnote on page \ref{footnote:sobolev}
for the definition of this family of Sobolev spaces.}
\[
	P : x^\alpha \scH^{k+m}(X/Z;V) \to x^\alpha \scH^{k}(X/Z;V) 
\]
for all $k, \alpha$.
\end{cor}
\begin{proof}
The interior symbol $\isym(P) = \isym(D)$ is invertible on $\scS^\ast (X/Z)$, since $D$ is elliptic. The boundary symbol
$\ssym(P) = \ssym(D + i\Phi) = \ssym(D) + i\pi^\ast \Phi$ is invertible on $\overline{\scT^\ast_{\pa X} (X/Z)}$ by symbolic commutativity, using 
Lemma \ref{L:fundamental} 
with $\alpha = \ssym(D)$, $\beta = \pi^\ast \Phi$ and $\Omega = \overline{\scT^\ast_{\pa X} (X/Z)}$.  $P$ is therefore fully elliptic, and by the theory of 
scattering pseudodifferential operators \cite{melrose1994spectral}, the claim follows.
\end{proof}

\begin{rem} Note how the compatibility of $\isym(P)$ and $\ssym(P)$ is satisfied.  Since $D$ is a family of operators of order $m > 0$, the leading term in the 
asymptotic expansion of $\ssym(P) = \ssym(D) + i\pi^\ast \Phi$ as $\abs{\xi} \to \infty$ is that of $\ssym(D)$, which grows like $\abs{\xi}^m$, 
whereas $\pi^\ast \Phi$ is constant.  In terms of the renormalized symbols and a choice of radial coordinate $\abs{\xi}$,
\[
	\ntsym(P) = \ntsym(D) + i \abs{\xi}^{-m} \pi^\ast \Phi
\]
and the latter term vanishes on $\scS^\ast (X/Z)$.
\end{rem}

\section{Reduction to the corner}\label{S:reduction}
By Corollary \ref{C:AS}, the index of $P$ is determined by the element in the odd K-theory
of $\pa(\overline{\scT^\ast (X/Z)})$ defined by the (renormalized) total symbol $\ntsym(P)$.  The remainder of our work 
consists of reducing this topological datum to one supported at the corner, $\scS^\ast_{\pa X} (X/Z)$. 
To this end, we will abstract the situation somewhat, in order to simplify the notation and 
clarify the concepts involved.  Thus we shall forget, for the time being, that our K-class is coming from the symbol of a family of pseudodifferential operators, 
as well as most of the structure of $\pa(\overline{\scT^\ast (X/Z)})$.

Let $M = \pa(\overline{\scT^\ast (X/Z)})$, and let $N = \scS^\ast_{\pa X} (X/Z)$ be the corner.  The important feature of $N$ is that it is a hypersurface, 
separating $M\setminus N$ into disjoint components $M_1 = \scS^\ast (\mathring{X}/Z)$ and $M_2 = \scT^\ast_{\pa X}(X/Z)$.  Actually, the fact that it is a corner 
is indistinguishable topologically, and we consider it just as a topological hypersurface in $M$.

We assume a trivialization of the line bundle $N_m$ has been chosen, so we identify $\ntsym(P)$ and $\tsym(P)$ and consider the index to be determined by the
element $[\pi^\ast V, \tsym(P)] \in K^1(M)$.  Also, for notational convenience, we will write $V$ instead of $\pi^\ast V$ for the remainder of this section.

Proposition \ref{P:ksetup} clarifies the fundamental symbolic structure of $P$.  We see that its symbol essentially consists of an invertible Hermitian term from $D$
over $M_1$ and an invertible skew-Hermitian term from $i\Phi$ over $M_2$, whose supports overlap in a neighborhood of the corner 
$N$.  The two terms are fundamentally coupled there, in that we cannot separate their supports via any homotopy in $\GL(V)$.  Also note that, were 
the total symbol either {\em entirely} Hermitian or {\em entirely} skew-Hermitian, it would be homotopic to the identity and $P$ would therefore have 
index 0.  Hence the nontriviality of $\ind(P)$ must be encoded by the coupling of the terms near the corner.  Proposition \ref{P:ktheory} confirms this, 
and identifies an element in $K^0(N)$ which captures this coupling.
	
\begin{prop} \label{P:ksetup}
$M$ is covered by two open sets $\tM_1$ and $\tM_2$ such that $\tM_1 \cap \tM_2 \cong N\times I$ where $I$ is a connected, open interval.  Furthermore,
\[
	[V,\tsym(P)] = [V,\chi A + i(1-\chi)B] \in K^1(M),
\]
where $A$ and $B$ are unitary\footnote{at least on $\supp \chi$ and $\supp (1-\chi)$, respectively.}, Hermitian sections of $\GL(V)$ such that 
$[A,B] = 0$ on $\tM_1 \cap \tM_2$, and where $\chi : M \to [0,1]$ is a cutoff function such that 
$\supp \chi \subset \tM_1$ and $\supp(1 - \chi) \subset \tM_2$.  The positive and negative eigenbundles\footnote{Meaning the bundles of positive and negative
eigenvectors.} of $A$ and $B$ coincide, respectively, with those of $\tsym(D)$ and $\pi^\ast\Phi$.
\end{prop}
\begin{proof}
As remarked at the end of Section \ref{S:callias}, $\tsym(P)$ is equal to $\ssym(D)$ on $M_1$ and to $\isym(D) + \phi\pi^\ast \Phi$ on $M_2$,
where $\phi \sim \abs{\xi}^{-m}$ is a nonnegative real-valued function vanishing on the closure of $M_1$.  In particular, 
$\phi \pi^\ast \Phi$ has the same $\pm$ eigenbundles as $\Phi$ wherever $\phi \neq 0$.

Since $\tsym(D)$ is invertible on $\overline{M_1} = M_1 \cup N$, by ellipticity, it must be invertible on a slightly larger 
neighborhood $\tM_1$.  We set $\tM_2 = M_2$, on which $\Phi$ is self-adjoint, invertible, and commutes with $\ssym(D)$ by the compatibility
assumption.  Shrinking either if necessary, we can assume that $\tM_1 \cap \tM_2 \cong N\times I$.  Let $\chi$ be a cutoff function with properties as above.

Recall that $C \in \GL(n,\C)$ is homotopic in $\GL(n,\C)$ to its unitarization $U(C)$ via
\[
	C_t = C\parens{t(\sqrt{C^\ast C})^{-1} + (1 - t)\id}\qquad U(C) = (C_t)_{t = 1}.
\]
Let $A$ and $B$ be the (generalized) unitarizations of $\tsym(D)$ and $\phi\pi^\ast \Phi$, respectively; thus $A$ is given by
\[
	A(p) = \begin{cases} U(\tsym(D)(p)) & \text{if $\tsym(D)(p)$ is invertible} \\
					0 & \text{otherwise,} \end{cases}
\]
and similarly for $B$.  
Note that while $A$ and $B$ are not necessarily continuous sections of $\End(V)$ (as $\tsym(D)$ may fail to be invertible
off of $\tM_1$ and $\phi\pi^\ast \Phi$ vanishes away from $\tM_2$), $\chi A$ and $(1-\chi)B$ {\em are} continuous and have 
support, respectively, where $\tsym(D)$ (resp. $\phi\pi^\ast \Phi$) is invertible.  

We claim that the homotopy $\sigma_t = (1-t)\tsym(P) + t\parens{\chi A + i(1-\chi)B}$ is through invertible endomorphisms.  
Indeed, at a general point $p \in M$,
\[
	\sigma_t(p) = \left[(1 - t)\tsym(D)(p) + t\chi(p)A(p)\right] + i\left[(1-t) \phi(p) \pi^\ast \Phi(p) + t (1-\chi)(p)B(p)\right].
\]
where the two bracketed terms commute with one another due to symbolic commutativity (where the latter is nonzero), and at least one term is invertible 
for any $p$ and all $t$.  Invertibility of $\sigma_t(p)$ is then immediate from Lemma \ref{L:fundamental}.
\end{proof}

In what follows we will identify $N\times I$ with the set $\tM_1 \cap \tM_2$, and denote its inclusion by $j : N\times I \hookrightarrow M$.
Note that over $N\times I$, $V$ splits as $V = V^+\oplus V^-$ into $\pm 1$ eigenbundles for $A$ (since $A$ is invertible here), and similarly 
$V = V_+\oplus V_-$ into $\pm 1$ eigenbundles for $B$.  Since $A$ and $B$ commute over $N\times I$, these splittings are compatible, giving
\[
	V_{|N\times I} \cong V^+_+\oplus V^-_- \oplus V^+_- \oplus V^-_+.
\]

The following makes use of the pushforward with respect to open embeddings in compactly supported K-theory, and also the Bott periodicity isomorphism
\[
	K^0(N) \stackrel{\cong}{\to} K^1(N\wedge S^1) = K^1_c(N\times I).
\]
\begin{prop}\label{P:ktheory}
Let $V_{|N\times I} \cong V^+_+\oplus V^-_- \oplus V^+_- \oplus V^-_+$ be the splitting into joint eigenbundles of $A$ and $B$ where $V^\pm$ denotes the 
$\pm$ eigenbundle of $A$ and $V_\pm$ the $\pm$ eigenbundle of $B$.  

Identify $[V^+_+] = [V^+_+] - [0] \in K^0(N)$ with its image in $K^1_c(N\times I)$ under the Bott isomorphism, and denote by $j_\ast([V^+_+])$ the image 
in $K^1(M)$ of $[V^+_+]$ under the pushforward $j_\ast : K^1_c(N\times I) \to K^1(M, M\setminus(N\times I))$ and the long exact sequence map 
$K^1(M, M\setminus(N\times I)) \to K^1(M)$.
Then
\[
	[V, \chi A + i(1-\chi)B] = j_\ast\parens{[V^+_+]} \in K^1(M),
\]

Alternatively, we could have used any of the bundles $V^\pm_{\pm/\mp}$, which are related by
\[
	j_\ast\parens{[V^+_+]} = j_\ast\parens{[V^-_-]} = - j_\ast\parens{[V^+_-]} = - j_\ast\parens{[V^-_+]}.
\]

\end{prop}
\begin{proof}
Let $\sigma = \chi A + i(1 - \chi)B$.  We first trivialize $\sigma$ away from $N\times I$, so that the K-class it defines is compactly 
supported\footnote{Recall that an element $\alpha \in K^1(M)$ in odd K-theory has support in a set $A$ if it is in the image of $K^1(M,M\setminus A)$ with respect
to the long exact sequence of the pair $(M,A)$, and can therefore be represented by an element $[V,\sigma]$ with $\sigma_{|M \setminus A} \equiv \id$.} 
in $N\times I$.

Let $\set{\rho_0,\ldots,\rho_4}$ be a partition of unity satisfying $\supp(\rho_i) \cap \supp(\rho_{j}) = \emptyset$ unless
$i = j$ or $i = j+1$, $\supp(\rho_0) \cap \supp(\chi) = \emptyset$, and, for $i \in \set{1,2,3}$, $\supp(\rho_i) \Subset N\times I = \wt M_1 \cap \wt M_2$.  
In particular, $\rho_0 \equiv 1$ away from $\wt M_1$ with $\supp(\rho_0) \subset \wt M_2$, and $\rho_4 \equiv 1$ away from $\wt M_2$ with $\supp(\rho_4) \subset \wt M_1$.
We claim there is a homotopy through invertible endomorphisms
\[
	\sigma \sim \sigma' = - \rho_0 \id + \rho_1 i B + \rho_2 A - \rho_3 i\id  - \rho_4 \id,
\]
so that $\sigma' \equiv - \id$ on the complement of $N\times I$.  Indeed, such a homotopy is given by
\[
	\sigma_t = \begin{cases} - t\rho_0 \id + (1-\chi) i B + \chi A & 0 \leq t \leq 1 \\
			         - \rho_0 \id + (2 - t)(1 - \chi) i B + (t - 1)\rho_1 i B + \chi A & 1 \leq t \leq 2 \\
				 - \rho_0 \id + \rho_1 i B + \chi A - (t - 2)(\rho_3 + \rho_4)i \id & 2 \leq t \leq 3 \\
				 - \rho_0 \id + \rho_1 i B + (4 - t)\chi A + (t - 3)\rho_2 A - (\rho_3 + \rho_4) i \id & 3 \leq t \leq 4 \\
				 - \rho_0 \id + \rho_1 i B + \rho_2 A - \rho_3 i \id - \rho_4\parens{(5 - t) i \id + (t - 4)\id} & 4 \leq t \leq 5 \end{cases}
\]
Note that for each $t$, $\sigma_t$ is invertible by Lemma \ref{L:fundamental} and the support conditions on $\set{\chi,\rho_0,\ldots,\rho_4}$.

Since $\sigma' \equiv - \id$ on $M \setminus (N\times I)$ (while we have trivialized $\sigma'$ by $-\id$ away from $N\times I$ instead of $+\id$, the two are equivalent
up to homotopy; indeed $\sigma' \sim - \sigma'$ for any clutching function), it is now evident that $[V,\sigma] = [V,\sigma']$ is in the image of 
$K^1(M,M\setminus(N\times I))$.

By identifying ends of the interval $I$, we see that $\sigma'$ defines a map
\[
	\sigma'_{|N\times S^1} = \begin{pmatrix} \sigma_1(\theta) & & & \\ & \sigma_2(\theta) & & \\ && \sigma_3(\theta) & \\ &&& \sigma_4(\theta) 
		\end{pmatrix}\qquad 
		\sigma_i : S^1 \to \C \backslash \set{0}.
\]
which is diagonal with respect to the splitting $V_{|N\times I} \cong V^+_+\oplus V^-_- \oplus V^+_- \oplus V^-_+$, with scalar entries (since $A$ and $B$ are unitary) 
independent of $N$, whose winding numbers are easily determined.

Indeed, by considering the effect of multiplication by $\sigma_i(t)$ as $\sigma'(t)$ passes from $-\id$, to $iB$, to $A$, to $-i\id$ and then back to $-\id$, 
we see that $\wn(\sigma_2) = \wn(\sigma_3) = \wn(\sigma_4) = 0$ and $\wn(\sigma_1) = -1$.  Thus there are homotopies
\[
	\sigma_i \sim \tilde \sigma_i \equiv 1, \quad i = 2, 3, 4, \quad \text{ and }\quad \sigma_1(\theta) \sim \tilde \sigma_1(\theta) = e^{-i\theta}.
\]
which, taken to be the diagonal elements of a matrix, define a homotopy $\sigma' \sim \tilde \sigma$.
Restricting to $N\times I$, we see
\[
	K^1_c(N\times I) = K^1(N\wedge S^1) \ni j^\ast[V,\tilde \sigma] 
	= [V^+_+\oplus V^-_-\oplus V^+_-\oplus V^-_+, e^{-i\theta}\oplus\id\oplus \id \oplus \id] = [V^+_+, e^{-i\theta}],
\]
which is just the image $[V^+_+]\cdot \beta \in K^1_c(N\times I)$ of $[V^+_+] \in K^0(N)$ under Bott periodicity, where 
$\beta = [\C, e^{-i\theta}] \in K^1(S^1) \cong K^1_c(I)$ is the Bott element.

Finally, since $[V,\sigma] = [V,\tilde \sigma] \in K^1(M))$ is in the image of $j_\ast : K^1_c(N\times I) \to K^1(M, M\setminus(N\times I)) \to K^1(M)$,
we obtain
\[
	[V,\sigma] = j_\ast ([V^+_+])
\]
as claimed.

Similar proofs, using initial trivializations to 
\[
	\sigma \sim \sigma'' = \rho_0 \id + \rho_1 i B + \rho_2 A + \rho_3 i \id + \rho_4 \id,
\]
\[
	\sigma \sim \sigma''' = \rho_0 \id + \rho_1 iB + \rho_2 A - \rho_3 i\id + \rho_4 \id,
\]
and 
\[
	\sigma \sim \sigma'''' = - \rho_0 \id + \rho_1 iB + \rho_2 A + \rho_3 i\id  - \rho_4 \id,
\]
give $[V,\sigma] = j_\ast([V^-_-])$, $[V,\sigma] = -j_\ast([V^-_+])$, and $[V,\sigma] = -j_\ast([V^+_-])$, respectively.
\end{proof}

\section{Results}\label{S:results}
We now present our main results.  To simplify notation, we drop the ``sc'' labels in the remainder of the paper, identifying 
$\scT^\ast (X/Z)$ with $T^\ast (X/Z)$ via a (non-canonical) isomorphism, which is unique up to homotopy.

\begin{thm} \label{T:main}
Given an elliptic family of scattering pseudodifferential operators $D \in \scP^m(X/Z;V)$ with Hermitian symbols, and a compatible
family of potentials $\Phi \in C^\infty(X; \End(V))$ as defined in section \ref{S:callias}, 
the family $P = D + i\Phi$ is fully elliptic, and extends to a Fredholm family with index satisfying
\[
	\ch(\ind(P)) = p_!(\ch(V^+_+)\cdot \pi^\ast \Td(\pa X/Z)),
\]
where $p_! : H^\mathrm{even}(S^\ast_{\pa X} (X/Z)) \to H^\mathrm{even}(Z)$ denotes integration over the fibers, $V^+_+ \to S^\ast_{\pa X} (X/Z)$ is the family 
of vector bundles corresponding to the jointly positive eigenvectors of $\tsym(D)_{|S^\ast_{\pa X} (X/Z)}$ and $\pi^\ast\Phi_{|\pa X}$, and $\tsym(D)$ is obtained
from $\ntsym(D)$ using any trivialization of $N_m$.
\end{thm}
\begin{rem} In the case of a single operator $P = D + i\Phi \in \scP^m(X; V)$, the index formula can be written
\[
	\ind(P) = \int_{S^\ast_{\pa X} X} \ch (V^+_+)\cdot\pi^\ast\Td(\pa X).
\]
\end{rem}

\begin{proof} 
By Corollary \ref{C:AS}, 
\[
	\ch(\ind(P)) = q_!\parens{\ch(\tsym(P))\cdot \pi^\ast \Td(X/Z)}.
\]
From Propositions \ref{P:ksetup} and \ref{P:ktheory}, 
$K^1(\pa(\overline{T^\ast (X/Z)})) \ni [\pi^\ast V, \tsym(P)] = j_\ast [V^+_+]$ with $[V^+_+] \in K^0(S^\ast_{\pa X} (X/Z)) \cong K^1_c(S^\ast_{\pa X} (X/Z) \times I)$, 
where $V^+_+$ is the jointly positive eigenbundle of $\tsym(D)$ and $\pi^\ast \Phi$. 

Now, since the Chern character is a natural mapping $\ch : K^\ast \to H^\ast$, we obtain 
\[
	\ch(\tsym(P)) = \ch(j_\ast [V^+_+]) = j_\ast \ch(V^+_+),
\]
where $j_\ast$ is the composition 
$H^{\mathrm{even}}(S^\ast_{\pa X} (X/Z)) \stackrel{\cong}{\to} H^{\mathrm{odd}}_c(S^\ast_{\pa X} (X/Z) \times I) \to H^{\mathrm{odd}}(\pa(\overline{T^\ast (X/Z)}))$.
Since $j_\ast\ch(V^+_+)$ is supported on $S^\ast_{\pa X} (X/Z)$, the integration over the fibers reduces to
\[
	\ch(\ind(P)) = p_!\parens{\ch(V^+_+)\cdot \pi^\ast \Td(X/Z)},
\]
where now $p : S^\ast_{\pa X} (X/Z) \to Z$.  Furthermore, since the Todd class is natural, $\pi^\ast \Td(X/Z)$ factors through 
$S^\ast_{\pa X} (X/Z) \to \pa X \hookrightarrow X$ (all over $Z$) and we obtain $\Td(X/Z)_{|\pa X} = \Td(\pa X/Z)$; this can alternatively be seen by taking 
a product metric at the boundary.
\end{proof}

An interesting case of the above is when the family $V = E\otimes F$ is a tensor product of vector bundles, with $D \in \scP^m(X/Z;E)$ and 
$\Phi \in C^\infty(X;\End(F))$.  In this case, $\ssym(D) \otimes 1$ and $i\otimes \Phi$ commute automatically, it is sufficient that $\Phi_{|\pa X}$ be
invertible and self-adjoint in order to be compatible with $D$.

\begin{thm} \label{T:main_product}
Given $D \in \scP^m(X/Z; E)$ elliptic with self-adjoint symbols, and a compatible potential $\Phi \in C^\infty(X; \End(F))$, the family 
$P = D\otimes 1 + i\otimes \Phi \in \scP^m(X/Z; E\otimes F)$ is fully elliptic, with Fredholm index satisfying
\[
	\ch(\ind(P)) = p_!(\ch(E_+)\cdot\ch(F_+)\cdot \pi^\ast \Td(\pa X/Z)),
\]
where $p_! : H^\mathrm{even}(S^\ast_{\pa X} (X/Z)) \to H^\mathrm{even}(Z)$ denotes integration over the fibers, and $\pi^\ast E = E_+\oplus E_-$ and 
$\pi^\ast F = F_+\oplus F_-$ are the splittings over $S^\ast_{\pa X}(X/Z)$ into positive and negative eigenbundles of $\sigma(D)$ and $\pi^\ast \Phi$,
respectively.
\end{thm}
\begin{rem} Note that the splitting of $F$ is actually coming from the base: $F_{|\pa X} = F_+ \oplus F_-$, and $\pi^\ast F = \pi^\ast F_+\oplus\pi^\ast F_-$.
\end{rem}
\begin{proof}
The proof is as above, noting that the splitting of $\pi^\ast V = \pi^\ast (E\otimes F)$ over $S^\ast_{\pa X} (X/Z)$ into
\[
	\pi^\ast V_{|S^\ast_{\pa X} (X/Z)} \cong V^+_+\oplus V^-_- \oplus V^+_- \oplus V^-_+
\]
corresponds to
\[
	\pi^\ast (E\otimes F)_{|S^\ast_{\pa X} (X/Z)} \cong (E_+ \otimes F_+)\oplus(E_-\otimes F_-)\oplus(E_+\otimes F_-)\oplus(E_-\otimes F_+),
\]
with $\pi^\ast E = E_+\oplus E_-$ and $\pi^\ast F = F_+\oplus F_-$ split into $\pm$ eigenbundles of $\tsym(D)$ and $\pi^\ast \Phi$, respectively.  
Then we note that in $K$-theory,
\[
	[E_+\otimes F_+] - [0] = ([E_+] - [0])\cdot([F_+] - [0]) \in K^0(S^\ast_{\pa X} (X/Z)).
\]
Since the pushforward $j_\ast : K_c^\ast(S^\ast_{\pa X} (X/Z)\times I)\to K^\ast(\pa(\overline{T^\ast (X/Z)}))$ behaves naturally with respect to products 
in $K$-theory, and since $\ch([E_+]\cdot[F_+]) = \ch(E_+) \cdot \ch(F_+)$, we have
\[
	\ch(\ind(P)) = p_!\parens{\ch(E_+)\cdot\ch(F_+)\cdot \pi^\ast\Td(\pa X/Z)},
\]
as claimed.
\end{proof}

\subsection{Dirac case}\label{S:dirac}
We further specialize to the case where $D = \dirac{D}$ is a family of (self-adjoint) Dirac operators, acting on sections of a family of Clifford modules $V$.  
In this case, our index formula further reduces to one over $T^\ast (\pa X/Z)$ rather than $S^\ast_{\pa X} (X/Z)$, and is given in terms of a 
related family of Dirac operators on $\pa X$.  

Let $\Cl(X/Z)$ denote the Clifford bundle $\Cl(\scT(X/Z), g)$. 
Suppose that $V\to X$ is a family of Clifford modules with unitary, skew-Hermitian action $\cl : \Cl(X/Z) \to \End(V)$, and a compatible {\em Clifford connection}
$\nabla : C^\infty(X;V) \to C^\infty(X;\scT^\ast (X/Z) \otimes V)$, i.e.
\[
	\nabla (\phi \cdot u) = \nabla^{\mathrm{LC}(g)}\phi \cdot u + \phi \cdot (\nabla u), \qquad \phi \in C^\infty(X;\Cl(X/Z)), u \in C^\infty(X;V),
\]
where $\nabla^{\mathrm{LC}(g)} : C^\infty(X;\Cl(X/Z)) \to C^\infty(X;\scT^\ast (X/Z) \otimes \Cl(X/Z))$ is the natural extension of the Levi-Civita connection 
to $\Cl(X/Z)$.
In analogy to the case of compact manifolds \cite{lawson1989spin}, these data lead to the construction of a canonical scattering Dirac operator
$\dirac{D} \in \scDiff^1(X/Z; V)$, defined 
at $p \in X$ by 
\[
	\dirac{D}_p = \sum_j \cl(e_j) \cdot \nabla_{e_j} : C^\infty_c(\mathring{X};V) \to C^\infty_c(\mathring{X};V), 
		\quad \set{e_j}_{j=1}^n \text{ an orthonormal basis for $\scT_p(X/Z)$},
\]
which is essentially self-adjoint with respect to the $L^2(X;V)$ pairing
\[
	(\dirac{D}u,v) = (u,\dirac{D}v).
\]
Note that $\ssym(\dirac{D})(p,\xi) = i \cl(\xi)\cdot $.  

There is a splitting of $V$ over $\pa X$ coming from the Clifford module structure.  To see this, recall the isomorphism
$$
	\Cl(\R^{n-1}) \cong \Cl^0(\R^n),
$$
$\Cl^0$ denoting the even graded part of the algebra, which is generated by $\R^{n-1} \ni e_i \mapsto e_i\cdot e_n$.  Similarly, given a choice of normal
section $\nu = x^2\pa_x : \pa X \to \scN (\pa X/Z)$, we have a bundle isomorphism
\[
	\Cl(\scT(\pa X/Z), g) \cong \Cl^0(\scT(X/Z),g)_{\pa X} = \Cl^0(X/Z)_{\pa X},
\]
and, by a choice of boundary defining function $x$, we can further identify $\Cl(\scT(\pa X/Z),g)$ and $\Cl(\pa X/Z) \equiv \Cl(T(\pa X/Z), h)$.
The family of Clifford modules $V_{|\pa X}$ has the structure of a $\Z/2\Z$-graded module over this subalgebra $\Cl(\pa X/Z)$.  Explicitly, if
\[
	V_{|\pa X} = V^0 \oplus V^1
\]
is the splitting according to $\pm 1$ eigenspaces of the Hermitian endomorphism $i \cl(\nu)$, then by the anticommutativity of 
$\scT(\pa X/Z)$ and $\scN(\pa X/Z)$ within $\Cl(X/Z)$, it follows that $\Cl(\pa X/Z)$ acts by graded endomorphisms, so
\[
	\Cl^j(\pa X/Z) : V^i \to V^{i+j}, \quad i,j \in \Z/2\Z.
\]
We denote this induced $\Cl(\pa X/Z)$ action by
\[
	\cl_0 : \Cl(\pa X/Z) \to \End_{\mathrm{gr}}(V^0\oplus V^1).
\]
If the Clifford connection $\nabla$ is the lift of a b connection (so $\nabla$ restricts to a connection on $\pa X$), we define the 
{\em induced boundary Dirac operator} $\dirac{\pa} \in \Diff^1(\pa X/Z; V)$ by
\[
	\dirac{\pa}_p = \sum_j \cl_0(e_j) \nabla_{e_j}, \quad \set{e_j}_{j=1}^{n-1} \text{ an orthonormal basis for $T_p(\pa X/Z)$.}
\]

Now assume $\Phi \in C^\infty(X; \End(V))$ is a compatible family of potentials.  In particular, symbolic commutativity at $\pa X$
implies that the positive/negative eigenbundles of $V_{|\pa X}$ with respect to $\Phi$ are themselves Clifford modules:
\[
	[\pi^\ast\Phi, \ssym(\dirac{D})] = [\pi^\ast \Phi,\cl(\cdot)] = 0 \in C^\infty(\scT^\ast (\pa X/Z); \End(V)) 
	\implies V_{|\pa X} = V_+\oplus V_-, \quad \Cl(X/Z)_{|\pa X} : V_\pm \to V_\pm.
\]
Furthermore, it is possible to choose the Clifford connection compatible with this splitting, so that the restriction of the connection
to $\pa X$ preserves $V_\pm$.

By symbolic commutativity, the splittings $V_{|\pa X} = V_+\oplus V_-$ and $V_{|\pa X} = V^0\oplus V^1$ are compatible, so we have
\[
	V_{|\pa X} = V^0_+ \oplus V^0_- \oplus V^1_+ \oplus V^1_-,
\]
with respect to which the induced boundary Dirac operator $\dirac{\pa}$ takes the form
\[
	\dirac{\pa} = \begin{pmatrix} 0 & \dirac{\pa}^-_+ & 0 & 0 \\ \dirac{\pa}^+_+ &0  &0 &0 \\ 0 & 0 & 0 & \dirac{\pa}^-_- \\ 0 & 0&\dirac{\pa}^+_- & 0 
				\end{pmatrix}.
\]
In particular, the operators $\dirac{\pa}^\pm_\pm$ are families of Dirac operators on $\pa X$, a fibration of closed manifolds over $Z$, whose principal symbols
are given by the induced Clifford action $\cl_0$, for instance,
\[
	\sigma(\dirac{\pa}^+_+) = i \cl_0(\cdot) : T^\ast (\pa X/Z) \to \Hom(\pi^\ast V_+^0, \pi^\ast V_+^1).
\]

It remains to show how the splitting of $\pi^\ast V_{|S^\ast_{\pa X} (X/Z)} = V^+ \oplus V^-$ into eigenbundles of $\tsym(\dirac{D})$ is related to the
splitting $V_{|\pa X} = V^0\oplus V^1$.

\begin{lem}\label{L:clifford_k_decomp}
We have
\[
	K^0(S^\ast_{\pa X} (X/Z)) \cong K^0_c(T^\ast (\pa X/Z))\oplus K^0(\pa X),
\]
with respect to which
\[
	[V^+_\pm] = [\pi^\ast V^0_\pm, \pi^\ast V^1_\pm, i\cl_0] + [V^1_\pm].
\]
\end{lem}
\begin{proof}
We can identify $S^\ast_{\pa X} (X/Z)$ with two copies of $\overline{T^\ast (\pa X/Z)}$, glued along their common boundary $S^\ast (\pa X/Z)$.  With this 
identification, the exact sequence of the pair $(S^\ast_{\pa X} (X/Z), \overline{T^\ast (\pa X/Z)})$ splits since there is an obvious retraction 
$S^\ast_{\pa X} (X/Z) \to \overline{T^\ast (\pa X/Z)}$ (projecting one hemisphere of each fiber onto the other).  Hence we have a split short exact sequence
\[
	0 \to K^0(S^\ast_{\pa X} (X/Z), \overline{T^\ast (\pa X/Z)}) \to K^0(S^\ast_{\pa X} (X/Z)) \to K^0(\overline{T^\ast (\pa X/Z)}) \to 0
\]
Along with the isomorphisms 
\[
	K^0(S^\ast_{\pa X} (X/Z), \overline{T^\ast (\pa X/Z)}) \cong K^0(\overline{T^\ast (\pa X/Z)}, S^\ast (\pa X/Z)) = K^0_c(T^\ast (\pa X/Z)),
\]
and $K^0(T^\ast (\pa X/Z)) \cong K^0(\pa X)$ (by contractibility of the fibers), we obtain
\[
	K^0(S^\ast_{\pa X} (X/Z)) \cong K^0_c(T^\ast (\pa X/Z))\oplus K^0(\pa X),
\]
as claimed.  

We will exhibit the decomposition of $[V^+_+]$ under this splitting; the case of $[V^+_-]$ is similar.  We claim that, as a vector bundle,
\[
	V^+_+ \cong \pi^\ast V^0_+ \cup_{i\cl_0(\cdot)} \pi^\ast V^1_+,
\]
that is, $V^+_+$ is isomorphic to the gluing of the vector bundles $\pi^\ast V^0_+ \to \overline{T^\ast (\pa X/Z)}$ and 
$\pi^\ast V^1_+ \to \overline{T^\ast (\pa X/Z)}$ 
via the clutching function $i \cl_0 : S^\ast (\pa X/Z) \to \Hom(\pi^\ast V^0_+,\pi^\ast V^1_+)$.  

To see this, observe that the two copies of $\overline{T^\ast (\pa X/Z)}$ in $S^\ast_{\pa X} (X/Z)$ retract, respectively, onto the images of $\pa X$ under the inward
and outward pointing conormal sections 
\[
	\nu : \pa X \to N^\ast (\pa X/Z) \subset S^\ast_{\pa X}(X/Z) \text{ and } - \nu: \pa X \to N^\ast (\pa X/Z) \subset S^\ast_{\pa X}(X/Z).
\]

Recall that $(V^+_+)_{\xi}$ is the positive eigenspace of Clifford multiplication $i\cl(\xi)$ at the point $\xi$, whereas 
$(\pi^\ast V^0_+)_{\xi}$ (resp.\ $(\pi^\ast V^1_+)_{\xi}$) is the positive (resp.\ negative) eigenspace 
of Clifford multiplication by the corresponding inward pointing normal, $i\cl(\nu(\pi(\xi)))$.  
Thus, over a point $\nu(p) \in S^\ast_{\pa X} (X/Z)$, we have 
\[
	(V^+_+)_{\nu(p)} = \pi^\ast (V^0_+)_{\nu(p)},
\]
while over the antipodal point $- \nu(p)$, we have 
\[
	(V^+_+)_{(- \nu(p))} = \pi^\ast (V^1_+)_{(- \nu(p))}, 
\]
since, for $v \in (V^+_+)_{(- \nu(p))}$, $v = i \cl( - \nu(p))v = - i \cl(\nu(p))v$.  

The bundle $V^+_+$ can therefore be identified with $\pi^\ast V^0_+$ and $\pi^\ast V^1_+$ over the inward and outward directed copies of 
$\overline{T^\ast (\pa X/Z)}$, respectively, and (since $\cl_0(\xi) = -i \cl(\xi) \sim \cl(\xi)$ on $\pi^\ast V^0$ when $\xi \in S^\ast (\pa X/Z)$), it is clear 
that $i\cl_0 \sim i\cl: S^\ast (\pa X/Z) \to \Hom(\pi^\ast V^0_+, \pi^\ast V^1_+)$ is the transition function gluing them together to produce $V^+_+$, which
finishes the claim.

Consider then the element
\[
	[V^+_+] = [V^+_+] - [\pi^\ast V^1_+] + [\pi^\ast V^1_+] \in K^0(S^\ast_{\pa X}(X/Z)).
\]
From the above, we see that $[V^+_+] - [\pi^\ast V^1_+]$ vanishes over the outward facing copy of $\overline{T^\ast (\pa X/Z)}$, and so maps to the element 
$[\pi^\ast V^0_+, \pi^\ast V^1_+, i \cl_0] \in K^0_c(T^\ast (\pa X/Z))$ in the decomposition above.  Clearly 
$K^0(\overline{T^\ast (\pa X/Z)}) \ni [\pi^\ast V^1_+] \cong [V^1_+] \in K^0(\pa X)$ under contraction along the fibers, and we therefore have
\[
	[V^+_+] \cong [\pi^\ast V^0_+, \pi^\ast V^1_+, i\cl_0] + [V^1_+],
\]
as claimed.
\end{proof}

The element $[\pi^\ast V^0_+, \pi^\ast V^1_+, i \cl_0] \in K^0_c(T^\ast \pa X)$ corresponds precisely to the symbol of $\dirac{\pa}^+_+$, and we obtain the following:

\begin{thm}\label{T:dirac}
Let $\dirac{D} \in \scDiff^1(X/Z; V)$ be a family of scattering Dirac operators acting on a family of Clifford modules $V \to X$, constructed from a lifted b 
Clifford connection, and suppose 
$\Phi \in C^\infty(X;\End(V))$ is a compatible potential.  Let $\dirac{\pa}^+_+ \in \Diff^1(\pa X/Z; V^0_+, V^1_+)$ be the graded part of the induced boundary Dirac
operator $\dirac \pa$ acting on the positive eigenbundle $(V_+)_{|\pa X} = V^0_+\oplus V^1_+$ of $\Phi_{|\pa X}$.
Then $P = \dirac{D} + i \Phi \in \scDiff^1(X/Z; V)$ extends to a family of Fredholm operators and
\[
	\ind(P) = \ind(\dirac{\pa}^+_+) \in K^0(Z).
\]
In particular, we have the index formula
\[
	\ch(\ind(P)) = p_!(\ch(\sigma(\dirac \pa ^+_+))\cdot \pi^\ast\Td(\pa X/Z))
\]
where $p_! : H_c^\mathrm{even}(T^\ast (\pa X/Z)) \to H^\mathrm{even}(Z)$ denotes integration over the fibers and, in the case of a single operator
\[
	\ind(P) = \ind(\dirac{\pa}^+_+) = \int_{T^\ast (\pa X)} \ch(\sigma(\dirac{\pa}^+_+))\cdot\pi^\ast\Td(\pa X).
\]
\end{thm}
\begin{proof}
We've seen that $[\pi^\ast V, \pi^\ast V, \tsym(P)] \in K^0_c(T^\ast (\mathring X/Z))$ is the image of $[V^+_+]$ under the composition
$K^0(S^\ast_{\pa X} (X/Z)) \stackrel{j_\ast}{\to} K^1(\pa(\overline{T^\ast (X/Z)})) \to K^0_c(T^\ast (\mathring{X}/Z))$.  In fact, it follows that this image 
is supported in a set homeomorphic to a small product neighborhood of the corner, $S^\ast_{\pa X} (X/Z) \times \R^2 \hookrightarrow T^\ast (\mathring{X}/Z)$, and that
\[
	[\pi^\ast V, \pi^\ast V, \tsym(P)] \cong \beta \cdot[V^+_+] \in K^0_c(S^\ast_{\pa X} (X/Z)\times \R^2)
\]
where $\beta \in K^0_c(\R^2)$ is the Bott element.

From Lemma \ref{L:clifford_k_decomp}, $[V^+_+]$ is a sum of two terms,
\[
	[V^+_+] = [\pi^\ast V^0_+, \pi^\ast V^1_+, \sigma(\dirac \pa ^+_+)] + [V^1_+],
\]
and we define $\sigma_1, \sigma_2 \in K^0_c(T^\ast (\mathring{X}/Z))$ to be the images of these terms.  As above, $\sigma_1$ and $\sigma_2$
are supported in our set $S^\ast_{\pa X} (X/Z)\times \R^2$, where they have the form
\[
	\sigma_1 \cong \beta \cdot [\pi^\ast V^0_+, \pi^\ast V^1_+, \sigma(\dirac \pa ^+_+)], \quad\text{and}\quad \sigma_2 \cong \beta \cdot [\pi^\ast V^1_+].
\]
To prove the theorem, it suffices to show that $\topind(\sigma_1) = \topind(\dirac \pa ^+_+)$ and that $\topind(\sigma_2) = 0$.  In fact, we will
show that $\sigma_2$ vanishes identically.

To see the latter, note that 
\[
	\sigma_2 \cong \beta \cdot [\pi^\ast V^1_+] = \beta \cdot \pi^\ast [V^1_+]
\]
where $\pi^\ast : K^0(\pa X) \to K^0(S^\ast_{\pa X}(X/Z))$.  It follows that $\sigma_2 \in K^0_c(T\ast (\mathring{X}/Z))$ is 
invariant with respect to the action of the rotation group $O(n)$ ($n = \dim (X/Z)$) on the fibers of $T^\ast (\mathring X/Z)$, since $O(n)$ acts
fiberwise on the first factor of $S^\ast_{\pa X}(X/Z) \times \R^2$, and $\pi^\ast [V^1_+]$ is constant on these fibers.

Thus $\sigma_2$ is obtained by pullback of an element 
\[
	\sigma_2' \in K^0_c(L), 
\]
where $L$ is the radial $\R_+$ bundle 
\[
	L = T^\ast (X/Z) / O(n) \to X.
\]
However, $K^0_c(L) \equiv 0$ for any such bundle, since it is equivalent to the (reduced) K-theory of $CX$, the cone on $X$, which is a contractible space.  
Therefore $\sigma'_2 = 0$ which implies $\sigma_2 = 0$.  

It remains to show $\topind(\sigma_1) = \topind(\dirac \pa^+_+)$.  Suppose we are given a K-oriented embedding of fibrations 
$T^\ast (X/Z) \hookrightarrow \R^{2N}\times Z$ coming from an embedding $X \to \R^N \times Z$.  
Let
\[
	g: S^\ast_{\pa X} (X/Z)\times \R^2 \hookrightarrow \R^{2N} \times Z
\]
be the induced embedding of our product neighborhood and let 
\[
	h = g_{|S^\ast_{\pa X}(X/Z) \times \set{(0,0)}}: S^\ast_{\pa X} (X/Z) \hookrightarrow \R^{2N} \times Z
\]
be the induced embedding of $S^\ast_{\pa X} (X/Z)$.  We denote the normal bundles of \mbox{$S^\ast_{\pa X} (X/Z)\times \R^2$} and $S^\ast_{\pa X} (X/Z)$
in $\R^{2N}\times Z$ by $N_g(S^\ast_{\pa X} (X/Z)\times \R^2)$ and $N_h(S^\ast_{\pa X} (X/Z))$, respectively, emphasizing the corresponding embeddings.
Of course
\[
	N_h(S^\ast_{\pa X} (X/Z)) = \R^2\times N_g(S^\ast_{\pa X} (X/Z)\times \R^2)
\]
since our neighborhood is a product.
	
Let 
\[
	f : T^\ast(\pa X/Z) \hookrightarrow S^\ast_{\pa X} (X/Z)
\]
be the open embedding onto the open inward facing open disk bundle as in Lemma \ref{L:clifford_k_decomp}, so
\[
	\sigma_1 \cong \beta \cdot f_\ast[\sigma(\dirac \pa ^+_+)],
\]
where $[\sigma(\dirac \pa^+_+)]$ will be shorthand for $[\pi^\ast V^0_+, \pi^\ast V^1_+, \sigma(\dirac \pa ^+_+)]$.
Finally, note that $h\circ f$ is a K-oriented embedding of $T^\ast (\pa X/Z)$ into a trivial Euclidean fibration which is homotopic (by stereographic projection) 
to the embedding induced by the map $\pa X \to X \to  \R^N \times Z$, and hence is suitable for computing $\topind(\dirac \pa ^+_+)$.

We will make use of the following general fact.  If $V \to B$ is an oriented complex vector bundle 
and $i: A \hookrightarrow B$ is an open embedding of spaces (which induces an open embedding $\tilde i : V_{|A} \to V$), then for any $\alpha \in K^\ast_c(A)$, we have 
\[
	i_\ast(\alpha) \cdot \Thom(V) = \tilde i_\ast \parens{\alpha\cdot \Thom(V_{|A})} \in K^\ast_c(V),
\]
where $\Thom(V) \in K^0_c(V)$ is the K-orientation class (Thom class) generating $K^0_c(V)$ as a module over $K^0(B)$.  As a consequence, we need only to show
that
\[
	\topind(\sigma_1) = \sigma_1 \cdot \Thom(N_g(S^\ast_{\pa X} (X/Z)\times \R^2)) \in K^0_c(\R^{2N}\times Z)
\]
is equivalent to 
\[
	\topind(\dirac \pa^+_+) = \tilde f_\ast \parens{[\sigma(\dirac \pa ^+_+)]\cdot\Thom(N_{h\circ f}(T^\ast (\pa X/Z)))} 
	= f_\ast [\sigma(\dirac \pa ^+_+)] \cdot \Thom(N_h(S^\ast_{\pa X} (X/Z))) \in K^0_c(\R^{2N} \times Z).
\]

However this follows immediately since, on the one hand, $\sigma_1 \cong \beta \cdot f_\ast[\sigma(\dirac \pa^+_+)]$;
and on the other hand, because 
$N_h(S^\ast_{\pa X} (X/Z)) = \R^2\times N_g(S^\ast_{\pa X} (X/Z)\times \R^2)$, we have
\[
	\Thom(N_h(S^\ast_{\pa X} (X/Z))) =  \beta \cdot \Thom(N_g(S^\ast_{\pa X} (X/Z)\times \R^2),
\]
by multiplicativity of the Thom class and the fact that the Thom class of a trivial $\R^2$ bundle is exactly $\beta$.
Thus, we obtain that $\ind(P) = \ind(\dirac \pa ^+_+)$ since $\topind(P) = \topind(\dirac \pa ^+_+)$, and the rest of the proof follows by taking the 
Chern character of both sides.
\end{proof}

Finally, we consider the product Dirac case; that is, assume $V = E \otimes F \to X$ where $\dirac{D} \in \scDiff^1(X/Z; E)$ acts on $E$
and the compatible potential $\Phi \in C^\infty(X;\End(F))$ acts on $F$, and $\dirac D\otimes 1 = \dirac D_{F}$ is obtained by equipping $E\otimes F$
with a tensor product connection.  We form the Callias-Anghel type family
\[
	P = \dirac{D} \otimes 1 + i \otimes \Phi \in \scDiff^1(X/Z; E\otimes F).
\]

As above, the Clifford module $E$ splits over the boundary into $E^0\oplus E^1$, with
\[
	\dirac \pa = \begin{pmatrix} 0 &\dirac\pa^- \\ \dirac\pa^+ & 0 \end{pmatrix},
\]
and $F_{|\pa X} = F_+\oplus F_-$ splits into positive and negative eigenbundles of $\Phi_{|\pa X}$.
\begin{thm}\label{T:dirac_product}
Let $P = \dirac{D}\otimes 1 + i \otimes \Phi \in \scDiff^1(X/Z; E\otimes F)$ as above.  Then $P$ extends to a Fredholm family with index
\[
	\ind(P) = \ind\parens{\dirac\pa^+_{F_+}}
\]
where $\dirac\pa^+_{F_+}$ is the twisted Dirac operator obtained by twisting $\dirac\pa^+ \in \scDiff^1(\pa X/Z; E^0, E^1)$ by $F_+$, the 
positive eigenbundle of $\Phi_{|\pa X}$.
\end{thm}
\begin{rem} In particular, when $X$ is an odd-dimensional spin manifold and $\dirac D$ is the (self-adjoint) spin Dirac operator 
(i.e.\ constructed using the fundamental representation of $\Cl(X)$ on spinors),
then $\dirac\pa^+ \in \Diff^1(\pa X; S^0, S^1)$ is the graded spin Dirac operator over the boundary, and we obtain
\[
	\ind(\dirac D \otimes 1 + i \otimes \Phi) = \int_{T^\ast \pa X} \ch(F_+) \cdot \hat{ \mathrm{A}}(\pa X),
\]
since $\ch(\sigma(\dirac \pa^+))\cdot\Td(\pa X) = \hat{\mathrm{A}}(\pa X)$ (compare to the formula obtained by R{\aa}de in \cite{rade1994callias}).
\end{rem}
\begin{proof}
The result follows from the previous one, after noting that $V^0_+$ and $V^1_+$ are given by $E^0\otimes F_+$ and $E^1\otimes F_+$, respectively, and that 
the clutching function
\[
	i\cl_0 : S^\ast(\pa X/Z) \to \Hom(\pi^\ast E^0\otimes F_+, \pi^\ast E^1\otimes F_+)
\]
is given by 
$\sigma(\dirac\pa^+)\otimes 1 = \sigma(\dirac \pa^+_{F_+})$.
\end{proof}


There are a few final remarks to be made:
\begin{itemize}
\item
First, regarding even/odd dimensionality:  in the case of (families of) Dirac operators, $P$ will only have a nonzero index when the dimension $\dim(X/Z)$ of the 
fiber is {\em odd}.  Since the index of $P$ reduces to the index of a family of {\em differential} operators on $\pa X \to Z$, it must vanish when 
$\dim(\pa X/Z) = \dim(X/Z) - 1$ is odd for the usual reason.  Because of this, previous literature on the subject was limited to the index problem on 
odd-dimensional manifolds, though we emphasize that, if $D$ is allowed to be pseudodifferential, $P$ may have nontrivial index even when $\dim(X/Z)$ is even.

\item
Our analysis of clutching data in the Dirac case, which related $[V^+_+] \in K^0(S^\ast_{\pa X} (X/Z))$ to the symbol 
$[\pi^\ast V^0_+, \pi^\ast V^1_+, \sigma(\dirac \pa ^+_+)] \in K^0_c(T^\ast (\pa X/Z)$ of an operator $\dirac \pa^+_+$ on $\pa X$ is equally valid when $D$ is
pseudodifferential.  Indeed, $V^+_+$ can always be written as the clutching of bundles $\pi^\ast V^0_+$ and $\pi^\ast V^1_+$ coming from the base, 
with respect to {\em some} clutching function $f$, and then $\ind(P) = \ind(\delta)$, where $\delta \in \Psi^m(\pa X/Z; V^0, V^1)$ is any elliptic 
pseudodifferential operator whose symbol $\sigma(\delta) = f$.  However, such a choice of $\delta$ is far from canonical without the additional structure of 
the Clifford bundles.
\end{itemize}

\section{Relation to previous results} \label{S:connection}
In \cite{anghel1993index}, N.\ Anghel generalized Callias' original index theorem to the following situation\footnote{See also \cite{rade1994callias} for an independently
obtained proof which addresses the Dirac product case as in section \ref{S:dirac}.} (adapted to our notation):  Let $X$ be a general odd-dimensional, 
non-compact, complete Riemannian manifold (with no particular structure assumed at infinity), with a Clifford module $V \to X$.  Let 
$\dirac{D} : C^\infty_c(X;V) \to C^\infty_c(X;V)$ be a self-adjoint Dirac operator, and $\Phi \in C^\infty(X;\End(V))$ a potential which is assumed to be uniformly 
invertible away from a compact set $K \Subset X$ and such that $[\dirac{D}, \Phi]$ is a uniformly bounded, 0th order operator (in particular, $\Phi$ commutes with 
Clifford multiplication).  First he proves that, for sufficiently large $\lambda > 0$,
\[
	P_\lambda = \dirac D + i \lambda \Phi \quad \text{is Fredholm,}
\]
essentially by showing that $P_\lambda$ and $P_\lambda^\ast$ satisfy what the author likes to call ``{\em injectivity near infinity}'' conditions:  
\[
	\norm{P_\lambda u}_{L^2} \geq c \norm{u}_{L^2} \quad \text{ for all $u \in C^\infty_c(X \setminus K; V)$}
\]
and similarly for $P_\lambda^\ast$.  In \cite{anghel1993abstract} Anghel shows how such conditions are equivalent to Fredholmness for self-adjoint Dirac operators, but 
it is easy to see that his proof generalizes to show that any differential operator $P$, which is injective near infinity along with its adjoint, extends to be Fredholm.

In any case, as in Section \ref{S:dirac}, $V$ splits over $X\setminus K$ into positive and negative eigenbundles of $\Phi$: $V_{|X \setminus K} = V_+\oplus V_-$, and
choosing a compact set $L \Subset X$ such that $K \subset \mathring{L}$ with $\pa L = Y$ a separating hypersurface (compare our earlier situation in which $Y = \pa X$), 
we have further compatible splitting $(V_\pm)_{|Y} = V_\pm^0\oplus V_\pm^1$ according to the decomposition $\Cl(Y) \cong \Cl(X)^0_{|Y}$, with 
$-i \cl(\nu) \equiv (-1)^i : V_\pm^i \to V_\pm^i$ where $\nu$ is a unit normal section.  Choosing appropriate connections, we can again construct an induced
Dirac operator on $Y$,
\[
	\dirac{\pa}_{Y} = \begin{pmatrix} 0 & \dirac{\pa}^-_+ & 0 & 0 \\ \dirac{\pa}^+_+ &0  &0 &0 \\ 0 & 0 & 0 & \dirac{\pa}^-_- \\ 0 & 0&\dirac{\pa}^+_- & 0 
				\end{pmatrix},
\]
and Anghel proves that
\[
	\ind({\dirac D + i \lambda \Phi}) = \ind({\dirac{\pa}_+^+}).
\]

His proof consists of index preserving deformations, along with the relative index theorem of Gromov and Lawson in \cite{gromov1983positive} (discussed further 
in \cite{anghel1993abstract}) to reduce to a product type Dirac operator $\widetilde{\dirac D}$ on a Riemannian product $Y \times \R$,
\[
	\widetilde{\dirac D} = i\cl(\nu)\frac{\pa}{\pa t} + \dirac \pa_+  + i \lambda \chi(t) : C^\infty_c(Y\times \R; V_+) \to C^\infty_c(Y\times \R;V_+)
\]
with equivalent index.  Here $\chi : \R \to [-1,1]$ is a smooth function such that $\chi \equiv -1$ near $- \infty$ and $\chi \equiv 1$ near $+ \infty$.  
Direct computation then shows that $\ind(\widetilde{ \dirac D}) = \ind(\dirac \pa_+^+)$.  

We point out that the steps in his proof could just as easily reduce to a {\em scattering} product (i.e. $Y \times \R$, but with locally Euclidean ends instead 
of cylindrical ones), with a scattering type Dirac operator
\[
	\widetilde{\dirac D}' = i\cl(\nu)\frac{\pa}{\pa t} + \frac{1}{t}\dirac \pa_+ + i \lambda \chi(t) : C^\infty_c(Y\times \R; V_+) \to C^\infty_c(Y\times\R;V_+)
\]
whose index is equivalent to $\ind(\dirac \pa_+^+)$ by our own Theorem \ref{T:dirac}.  This is really overkill in this case, since the index of $\widetilde{\dirac D}$
is determined simply enough; however, it raises the point that scattering-type infinite ends ($\pa L \times [0, \infty)$, where $L \Subset X$ as above) 
may be utilized for the purpose of computing the index of Dirac operators satisfying Anghel's Fredholm conditions (injectivity near infinity for $P$ and its adjoint).
The author anticipates that cutting and gluing constructions, similar to those used in \cite{gromov1983positive} to prove the relative index formula, 
may be able exhibit such equivalences for arbitrary Fredholm differential operators satisfying injectivity near infinity conditions.

\bibliographystyle{amsalpha}
\bibliography{references}

\end{document}